\documentclass[a4paper,11pt]{article}
\usepackage[utf8]{inputenc}
\usepackage[english]{babel}

\usepackage{amsmath}
\usepackage{amsthm}
\usepackage{amssymb}
\usepackage{bm}
\usepackage{tikz}
\usepackage{caption}
\usepackage{subcaption}
\usepackage{rotating}

\usepackage{stmaryrd}

\newtheorem{theorem}{Theorem}[section]
\newtheorem{lemma}[theorem]{Lemma}
\theoremstyle{definition}

\newtheorem{algorithm}[theorem]{Algorithm}
\theoremstyle{remark}

\tikzstyle{vertex}=[minimum size=3pt,fill,draw, circle,inner sep=0pt]
\tikzstyle{vertexsq}=[minimum size=3pt,fill,draw, rectangle,inner sep=0pt]

\DeclareMathOperator{\Lapl}{\triangle}
\DeclareMathOperator{\Ker}{ker}

\newcommand{\ud}{\mathrm{d}}

\title{Cross-Points in Domain Decomposition Methods with a Finite
  Element Discretization} \author{Martin J. Gander\thanks{Université
    de Genève, Section de Mathématiques, \texttt{Martin.Gander@unige.ch} }, Kévin
  Santugini\thanks{Institut Polytechnique de Bordeaux, Institut Mathématiques de
Bordeaux, CNRS UMR5251, \texttt{Kevin.Santugini-Repiquet@ipb.fr}}}
\date{\today}

\begin{document}
\maketitle

\begin{abstract}
Non-overlapping domain decomposition methods necessarily have to
exchange Dirichlet and Neumann traces at interfaces in order to be
able to converge to the underlying mono-domain solution. Well known
such non-overlapping methods are the Dirichlet-Neumann method, the
FETI and Neumann-Neumann methods, and optimized Schwarz methods. For
all these methods, cross-points in the domain decomposition
configuration where more than two subdomains meet do not pose any
problem at the continuous level, but care must be taken when the
methods are discretized. We show in this paper two possible approaches
for the consistent discretization of Neumann conditions at
cross-points in a Finite Element setting.
\end{abstract}

\section{Introduction}

Domain decomposition methods~(DDMs) are among the best parallel solvers for
elliptic partial differential equations, see the books
\cite{Smith:1996:DPM,Quarteroni:1999:DDM, Toselli:2004:DDM} and
references therein. While classical Schwarz methods only exchange
Dirichlet information from subdomain to subdomain, and converge
because of overlap, non-overlapping methods like Dirichlet-Neumann,
FETI, Neumann-Neumann and optimized Schwarz methods~(OSMs) also exchange
Neumann traces, or combinations of Dirichlet and Neumann traces
between subdomains. In a general decomposition of a domain
$\Omega\subset\mathbb{R}^2$ into non-overlapping subdomains
$(\Omega_i)_{1\leq i\leq I}$, naturally cross-points arise. Such
cross-points, where more than two subdomains meet, do not pose any
problem in a continuous variational setting, but as soon as one
introduces a finite dimensional approximation, the discretization of a
Neumann condition over a cross-point does not follow naturally. The
earliest paper dedicated to cross-points dates, to our knowledge, back
to 1986: in
\cite{Dryja.Proskurowsky.Widlund:1986:MethodDomainDecompositionCrosspoints},
a Dirichlet-Neumann method is presented for domain decompositions with
cartesian topology that can be colored with only two colors.  Boundary
points, including cross-points, are part of the Neumann subdomains,
and all Neumann subdomains are coupled at cross-points, while
Dirichlet subdomains are fully decoupled.
In~\cite{Bialecki.Dryja:2008:NonoverlappingDDCrosspointsCollocation},
a Krylov accelerated DDM to compute the collocation solution of the
Poisson equation in a square with Hermite finite elements is studied.
There are four subdomains in a $2\times2$ grid configuration, thus
involving a cross-point, and theoretical convergence estimates are
provided. The FETI-DP
algorithm~\cite{Farhat:2001:DPU,Klawonn:2002:DPU} modifies the FETI
algorithm~\cite{Mandel:1996:BDP} at cross-points by replacing the dual
variables by primal ones and thus avoiding the problem of Neumann
conditions there. Similarly, strong coupling at cross-points is also
proposed
in~\cite{Bendali.Boubendir:2006:NonOverlappingDDM,Boubendir.Bendali.Fares:2008:CouplingNonOverlappingDDM}
for nodal finite elements.
In~\cite{Gander.Kwok:2012:BestRobinParameters}, it was shown for
optimized Schwarz methods (OSMs) in an algebraic setting that
optimized Robin parameters scale differently at cross-points, namely
like $O(1/h)$, in contrast to $O(1/\sqrt{h})$ at interface points
which are not cross-points, see also~\cite{loisel2013condition} for
condition number estimates in the presence of cross-points.  Cross
points can also be handled in the context of mortar methods, and in
very special symmetric configurations, it is actually possible for
cross-points not to pose any problems, 
see~\cite{Gander.Kwok:ApplicabilityLionsEnergyEstimates}. The cross-point problem can be avoided
entirely when using cell-centered finite volume discretizations,
because they do not contain cross-points at the discrete 
level, see~\cite{Cautres.Herbin.Hubert:2004:LionsDomainDecompositionMethod}
for the convergence of the cell-centered finite volume Optimized
Schwarz method with Robin transmission
conditions; see~\cite{Halpern.Hubert:FiniteVolumeVentcellSchwarz} for the convergence of the
cell-centered finite volume Optimized Schwarz with Ventcell
transmission conditions in the absence of cross-points; 
and~\cite{Gander.Kwok.Santugini:EnergyProofDiscreteOSMCellCenteredFV} for
the extension of the convergence proof to symmetric positive
definite transmission operators even in the presence of cross-points.

We describe in this paper in detail two approaches to exchange Neumann
traces over cross points in a finite element setting for two
dimensional problems: the {\em auxiliary variable method}, and {\em
  complete communication}. The auxiliary variable method keeps in
addition to the primal unknowns also auxiliary unknowns representing
interface data in each subdomain. These auxiliary variables permit a
consistent discretization of the Neumann traces at cross points while
only communicating with neighboring domains sharing a boundary of
non-zero one-dimensional measure. As a first main result, we show that
with auxiliary variables, one can prove convergence of the discretized
domain decomposition algorithm using energy estimates, which is not
possible for finite element discretizations with cross-points
otherwise~\cite{Gander.Kwok:ApplicabilityLionsEnergyEstimates}. A disadvantage of the
auxiliary variables is that they are not necessarily converging to a
limit, but this does not affect the convergence of the primal unknowns
in the iteration. The complete communication method needs to exchange
information with all subdomains touching at cross points, also those
which touch only at a point, in order to have a consistent
discretization of Neumann conditions. Our second main result is to
show how to determine among the many possible splittings of Neumann
traces one that minimizes oscillation.

Our paper is organized as follows: in~\S\ref{sect:DiscreteOSM}, we
describe on the concrete example of an OSM why the discretization of
the Neumann part of the transmission condition is ambiguous at
cross-points. In~\S\ref{sect:AuxiliaryVariables}, we present the first
approach on how to transmit Neumann information near cross-points
using auxiliary variables, and give a general convergence proof for a
non-overlapping OSM discretized by finite elements with
cross-points. In~\S\ref{sect:CompleteCommunications}, we describe how
Neumann information can be transmitted near cross-points by
communicating among all subdomains sharing the cross point, and we
propose a specific method minimizing oscillation.  After our
conclusions in~\S\ref{sect:Conclusions}, we show in Appendix
\ref{sect:Lumping} that instead of using higher order, so called
Ventcell transmission conditions, see for
example~\cite{Japhet:1997:CAL, Japhet:1998:OKV, Chevalier:1998:SMO,
  Japhet:2000:BIC, Japhet:2000:OO2, Gander:2000:OSM, Gander:2006:OSM},
one can algebraically naturally obtain such conditions from Robin
conditions using mass lumping techniques in a finite element setting.
This avoids the need for discretizing higher order differential
operators in the tangential direction, and even works at cross-points,
which is our third important result. 

\section{The discrete Optimized Schwarz
  Method}\label{sect:DiscreteOSM}

For the elliptic problem $\mathcal{L}u=f$ in $\Omega$, and a
non-overlapping decomposition $(\Omega_i)_{1\leq i\leq I}$, the OSM
with Robin transmission conditions at the continuous level is (see for
example~\cite{Gander:2006:OSM})
\begin{algorithm}[OSM]\label{algo:ContinuousOSM}\quad\\\quad\vspace{-3ex}
\begin{enumerate}
\item Set $p>0$.
\item Start with an initial guess $u_i^0$ in each subdomain $\Omega_i$.
\item Until convergence, compute in parallel the unique solution
  $u_i^{n+1}$ to
\begin{eqnarray}\label{solve}
\mathcal{L}u_i^{n+1}&=&f\;\text{in $\Omega_i$},\\\label{TransmissionCond}
\frac{\partial u^{n+1}_i}{\partial\bm{n}_{ii'}}+pu^{n+1}_i&=&\frac{\partial u^{n}_{i'}}{\partial\bm{n}_{ii'}}+pu^{n}_{i'}\;\text{on $\partial\Omega_i\cap\partial\Omega_{i'}$}.
\end{eqnarray}
\end{enumerate}
\end{algorithm}
\noindent In a variational formulation of Algorithm
\ref{algo:ContinuousOSM}, 
cross-points do not pose any problems, since
they have measure zero.  In a finite dimensional approximation
however, using for example finite elements, the Neumann part of the
Robin transmission conditions is only known as a variational quantity,
as an integral over the edges connected to the cross-point. When
discretizing OSM (or any DDM), there are two guiding principles:
\begin{enumerate}
\item The discrete mono-domain solution should be a fixed point of the
  discrete OSM.
\item The discrete OSM should have a unique fixed point.
\end{enumerate}
We show in this section that it is not completely straightforward to
follow these two principles when cross-points are present.

\subsection{Geometric setting and notation}\label{subsect:OSM-FEMSetting}

Let $\mathcal{T}$ be a polygonal mesh of $\Omega\subset\mathbb{R}^2$. Let
$(\Omega_i)_{1\leq i\leq I}$ be a non-overlapping domain decomposition
of the domain $\Omega$. We assume that the subdomains $\Omega_i$ are
polygonal, and that each cell of $\mathcal{T}$ is included in
exactly one subdomain. Let $\mathcal{T}_i$ be the restriction of the
mesh $\mathcal{T}$ to $\Omega_i$, and denote by $\bm{x}_j$ the vertices
of the mesh $\mathcal{T}$. We consider a finite element space
$\mathcal{P}(\mathcal{T})$ subset of $H_0^1(\Omega)$ with the
following properties:
\begin{enumerate}
\item There is exactly one degree of freedom at each vertex of
  $\mathcal{T}$ for $\mathcal{P}(\mathcal{T})$.
\item For any edge $\lbrack\bm{x}_j\bm{x}_{j'}\rbrack$ of
  $\mathcal{P}(\mathcal{T})$ and for any $u$ in
  $\mathcal{P}(\mathcal{T})$, $u(\bm{x}_j)=0$ and
  $u(\bm{x}_{j'})=0$ implies $u$ vanishes on the entire edge
  $\lbrack\bm{x}_j\bm{x}_{j'}\rbrack$.
\end{enumerate}
Both these conditions are satisfied for $P_1$ elements on triangular
meshes and $Q_1$ elements on cartesian ones. We define
$\mathcal{P}(\mathcal{T}_i):=\{u_{\vert\Omega_i}\vert u\in
\mathcal{P}(\mathcal{T})\}$. We denote the hat functions by $\phi_j$,
i.e. the unique function in $\mathcal{P}(\mathcal{T})$ such that
\begin{equation*}
  \phi_j(\bm{x}_{j'})=\begin{cases}
    1&\text{if $j=j'$,}\\
    0&\text{if $j\neq j'$,}
  \end{cases}
\end{equation*}
and by $\phi_{i;j}$ we denote $(\phi_j)_{\vert\Omega_i}$. We will
systematically use for subdomain indices the letter $i$, and separate
it from nodal indices $j$ using a semicolon. The discretized OSM
operates then on the space
\begin{equation*}
  V:=\bigotimes_{i=1}^N\mathcal{P}(\mathcal{T}_i).
\end{equation*}
Since a node located on a subdomain boundary may belong to more than
one subdomain, we use the index $i$ in $\bm{x}_{i;j}$ to distinguish
degrees of freedom located at the same node but belonging to different
subdomains.

\subsection{Discretization of Robin transmission conditions}
  \label{subsect:DiscretizationRobin}

The discrete Neumann boundary condition must be computed variationally
in a FEM setting, see for example~\cite[p.3,
  Eq. (1.7)]{Toselli:2004:DDM}. Near cross-points, the Neumann
boundary condition is like an integral over both edges that are
adjacent to the cross-point and belonging to the boundary of the
subdomain. As there is no canonical way to split that variational
Neumann boundary condition, it is not clear how we should split that
quantity when it comes to transmitting Neumann information
between adjacent subdomains near cross points. Any splitting
should satisfy the two guiding principles listed at the 
beginning of~\S\ref{sect:DiscreteOSM}.

To investigate this problem, it suffices to study the case of the
elliptic operator $\mathcal{L}:=\eta-\Lapl$, $\eta>0$ in Algorithm
\ref{algo:ContinuousOSM}. Following finite element principles, we
should solve for every subdomain $\Omega_i$ at every new iteration
$n+1$
\begin{equation}\label{FEMOSM}
\eta\int_{\Omega_i}u_i^{n+1}\phi_{i;j}
+\int_{\Omega_i}\nabla u_i^{n+1}\cdot\nabla\phi_{i;j}
+p\int_{\partial\Omega_i}u_i^{n+1}\phi_{i;j}\ud\sigma(\bm{x})=f_{i;j}+g_{i;j}^{n+1}
\end{equation}
for all $j$ such that $\bm{x}_{i;j}$ is a node of mesh $\mathcal{T}$
located in $\overline{\Omega}_i$, in order to
find the new finite element subdomain solution approximation
$u_i^{n+1}=\sum_ju_{i;j}^{n+1}\phi_{i;j}$.  The data $g_{i;j}^{n+1}$
needs to be gathered from neighboring subdomains, satisfying~\eqref{TransmissionCond} variationally.  We denote by the
matrix $\mathbf{A}_i$ the sum of the mass and stiffness contributions
corresponding to the interior equation $\eta-\Lapl$ in each subdomain
$\Omega_i$,
\begin{equation}\label{eq:InteriorMatrixA}
  A_{i;j,j'}:=\eta
  \int_{\Omega_i}\phi_{i;j}(\bm{x})\phi_{i;j'}(\bm{x})\ud\bm{x}
+\int_{\Omega_i}\nabla\phi_{i;j} (\bm{x})
    \nabla\phi_{i;j'}(\bm{x})\ud\bm{x}.
\end{equation}
The matrix $\mathbf{B}_i^{\textrm{cons}}$ contains the boundary contribution
$p\int_{\partial\Omega_i}u_i^{n+1}\phi_{i;j}\ud\sigma(\bm{x})$,
including the Robin parameter $p$: if the finite elements are linear
on each edge, which holds for $Q_1$ and $P_1$ elements, we have the
consistent interface mass matrix
\begin{equation}\label{eq:UnlumpedBoundaryMatrix}
B^{\textrm{cons}}_{i;j,j'}:=
\begin{cases}
  \frac{p}{3}\sum_{j''}\lvert\bm{x}_{i;j}-\bm{x}_{i;j''}\rvert
    &\text{if $j'=j$ and $\bm{x}_{i;j}$ lies on $\partial\Omega_i$},\\
  \frac{p}{6}\lvert\bm{x}_{i;j}-\bm{x}_{i;j'}\rvert
    &\text{if $\lbrack\bm{x}_{i;j}\bm{x}_{i;j'}\rbrack$ is an 
     edge of $\partial\Omega_i$},\\
  0&\text{otherwise},
\end{cases}
\end{equation}
where the sum is taken over all $j''\neq j$ such that
$\lbrack\bm{x}_j\bm{x}_{j''}\rbrack$ is a boundary edge of $\mathcal{T}_i$.
A lumped version of the interface mass matrix $\mathbf{B}_i^{\textrm{cons}}$ is
\begin{equation}\label{eq:LumpedBoundaryMatrix}
   B^{\textrm{lump}}_{i;j,j'}:=
  \begin{cases}
\frac{p}{2}\sum_{j''}\lvert\bm{x}_{i;j}-\bm{x}_{i;j''}\rvert&\text{if
  $j=j'$ and $\bm{x}_{i;j}$ lies on $\partial\Omega_i$},\\
0&\text{otherwise},
\end{cases}
\end{equation}
where again the sum is taken over all $j''\neq j$ such that
$\lbrack\bm{x}_j\bm{x}_{j''}\rbrack$ is a boundary edge of $\mathcal{T}_i$.
We explain in Appendix \ref{sect:Lumping} why using a lumped interface
mass matrix $\mathbf{B}^{\textrm{lump}}_i$ leads to faster convergence than using a
consistent mass matrix $\mathbf{B}_i$, by interpreting the lumping process at
the continuous level as introducing a higher order term in the
transmission condition, see also~\cite{dolean2011can}. This higher
order term can even be optimized using a new concept of overlumping we
will introduce. Note that in the context of discrete duality finite volume
methods, it was shown
in~\cite{Gander.Hubert.Krell:2013:OptimizedSchwarzDDFV} that the
consistent mass matrix can even completely destroy the asymptotic
performance of the optimized Schwarz method, even without cross-points.
This is however not the case for the finite element discretizations we
consider here.

Using the matrix notation we introduced, we have to solve at each
Schwarz iteration the to (\ref{FEMOSM}) equivalent matrix problem
\begin{equation}\label{eq:OSMStepSystem}
  (\mathbf{A}_i+\mathbf{B}^{\textrm{lump}}_i)\bm{u}_i^{n+1}=\bm{f}_{i} +\bm{g}^{n+1}_i,
\end{equation}
where the vector $\bm{g}^{n+1}_i$ is zero at interior nodes of
$\Omega_i$ and contains the values $\bm{g}_{i,i'}^n$ transmitted from
the neighboring subdomains $\Omega_{i'}$ on the interface nodes of
$\Omega_i$. The computation of $\bm{f}_{i}$ and $\bm{g}^{n+1}_i$
should be done in such a way that the two guiding principles listed at
the beginning of \S\ref{sect:DiscreteOSM} are satisfied. At the
continuous level, $\bm{f}_i$ would just be the restriction of $f$ to
$\Omega_i$, and hence, if the continuous function $f$ is known, one
can set
\begin{equation*}
  f_{i;j}:=\int_{\Omega_i}f(\bm{x})\phi_{i;j}\ud\bm{x}.
\end{equation*}
If only $\bm{f}$ is known, then one has to choose $\bm{f}_{i}$ in such
a way that the $j$th component of $\bm{f}$ satisfies
$f_{j}=\sum_if_{i;j}$ where the sum happens over all indices $i$ such
that $\bm{x}_j$ belongs to $\overline{\Omega}_i$.  For the transmitted
values $\bm{g}^{n}_{i,i'}$ with a finite element discretization, the
Neumann contribution is defined by a variational problem. At the
continuous level, if $(\eta-\Lapl)u_i=f$ inside $\Omega_i$, we have by
Green's formula
\begin{equation}
\int_{\partial\Omega_i}\frac{\partial u_i}{\partial\bm{n}_i}v=
\eta\int_{\Omega_i}uv
+\int_{\Omega_i}\nabla u\nabla v
-\int_{\Omega_i}fv.
\end{equation}
This formula can be used to define discrete Neumann boundary conditions:
for $\bm{x}_{i;j}$ a vertex of the fine mesh located on
$\partial\Omega_i$, we define
\begin{equation}\label{eq:NeumannDiscrete}
\mathcal{N}_{i;j}(u_i):=
\eta\int_{\Omega_i}u_i\phi_{i;j}
+\int_{\Omega_i}\nabla u_i\nabla \phi_{i;j}
-\bm{f}_{i;j}.
\end{equation}
At the discrete level, the no Neumann jump condition satisfied by
the discrete mono-domain solution is given by
$\sum_{i}\mathcal{N}_{i;j}(u_i)=0$ where the sum is over all $i$
such that $\bm{x}_j$ is a boundary vertex of $\mathcal{T}_i$.  For interface points that
belong to exactly two subdomains $\overline{\Omega}_i$ and
$\overline{\Omega}_{i'}$, the Robin update is not ambiguous and we set
\begin{equation}\label{eq:RobinUpdate}
  g^{n}_{i,i';j}:=-\mathcal{N}_{i';j}(u^{n}_{i'})
    +\frac{p}{2}u_{i';j}^{n}\sum_{j'}\lvert\bm{x}_{i;j}-\bm{x}_{i;j'}\rvert,
\end{equation}
where the sum is over all $j'$ such that
  $\lbrack\bm{x}_{j}\bm{x}_{j'}\rbrack$ is a boundary edge of both $\mathcal{T}_i$
  and $\mathcal{T}_{i'}$.
The $g^{n}_{i,i';j}$ must be sent by subdomain $\Omega_{i'}$ to subdomain
$\Omega_{i}$, and then $g^{n+1}_{i;j}=g^{n}_{i,i';j}$, since there is
only one contribution from the unique neighbor $\Omega_{i'}$.

\subsection{Ambiguity of the Robin update at cross-points}\label{subsect:CrossPointsAmbiguity}

To see why the Robin update~\eqref{eq:RobinUpdate} can not be used 
at cross points, consider as an example the cross point $\bm{x}_1$
belonging to subdomain $\Omega_1$ shown in Figure~\ref{fig:CrossPoint}.
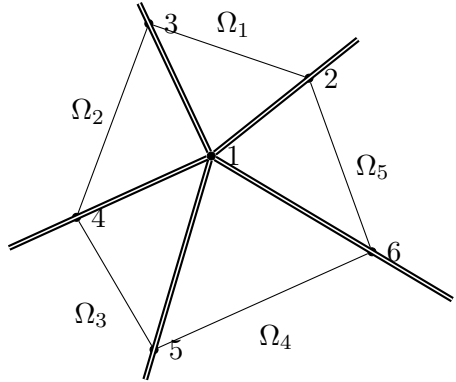
\begin{figure}
\centering
\begin{tikzpicture}
\node[vertex,label=right:$1$] (O) at (0,0.5) {} ;
\node[vertex, label=right:$2$] (A) at (50:2) {};
\node[vertex, label=right:$3$] (B) at (110:2.4) {};
\node[vertex, label=right:$4$] (C) at (190:1.8) {};
\node[vertex, label=right:$5$] (D) at (250:2.2) {};
\node[vertex, label=right:$6$] (E) at (340:2.25) {};
\draw[double,thick] (O) -- (barycentric cs:O=-0.25,A=0.75);
\draw[double,thick] (O) -- (barycentric cs:O=-0.1,B=0.75);
\draw[double,thick] (O) -- (barycentric cs:O=-0.25,C=0.75);
\draw[double,thick] (O) -- (barycentric cs:O=-0.1,D=0.75);
\draw[double,thick] (O) -- (barycentric cs:O=-0.25,E=0.75);

\node at (barycentric cs:O=-0.2,A=0.5,B=0.5) {$\Omega_1$};
\node at (barycentric cs:O=-0.2,B=0.5,C=0.5) {$\Omega_2$};
\node at (barycentric cs:O=-0.2,C=0.5,D=0.5) {$\Omega_3$};
\node at (barycentric cs:O=-0.2,D=0.5,E=0.5) {$\Omega_4$};
\node at (barycentric cs:O=-0.2,E=0.5,A=0.5) {$\Omega_5$};

\draw (A) -- (B) -- (C) -- (D) -- (E) -- (A) ;
\end{tikzpicture}
\caption{Example of a cross point in the decomposition}
\label{fig:CrossPoint}
\end{figure}
Following~\eqref{eq:RobinUpdate}, to compute $\bm{g}^{n+1}_{1}$ at
cross-point $\bm{x}_1$, one would intuitively set
\begin{equation*}
\begin{split}
  g^{n+1}_{1;1}&=-\mathcal{N}_{2;13} (u^{n}_2)
    +\frac{p}{2}\lvert\bm{x}_1-\bm{x}_3\rvert u_{2;1}^{n}\\
    &\phantom{=}-\mathcal{N}_{5;12}
    (u^{n}_5)+\frac{p}{2}\lvert\bm{x}_1-\bm{x}_2\rvert u_{5;1}^{n},
\end{split}
\end{equation*}
where $\mathcal{N}_{2;13}$ is the part of $\mathcal{N}_{2}$
located on edge $\lbrack\bm{x}_1\bm{x}_3\rbrack$, and likewise for
$\mathcal{N}_{5;12}$. Unfortunately, at the discrete level, the
Neumann contributions of $u^{n}_2$ and $u^{n}_5$ at $\bm{x}_{1}$ are only
known as an integral over the edges coming from $\bm{x}_{1}$. We
cannot distinguish the contribution of each edge to the Neumann
conditions $\mathcal{N}_{2} (u^{n}_2)$ and
$\mathcal{N}_{5} (u^{n}_5)$. We only know that
\begin{equation*}
  \mathcal{N}_{2} (u^{n}_2)=\mathcal{N}_{2;13} (u^{n}_2)
  +\mathcal{N}_{2;14} (u^{n}_2),\quad
  \mathcal{N}_{5} (u^{n}_5)=\mathcal{N}_{5;12} (u^{n}_5)
  +\mathcal{N}_{5;16} (u^{n}_5).
\end{equation*}
When transmitting the Robin condition at a cross point, the Neumann
contribution must be split across each edge in such a way that the
discrete mono-domain solution remains a fixed point of the optimized
Schwarz method, see principle 1 at the beginning of
\S\ref{sect:DiscreteOSM}. The discrete mono-domain solution satisfies
\begin{equation}\label{eq:MonodomainDiscreteSolutionsProperty}
u_{i;j}=u_{i';j}\quad \text{for all $i'$ with
  $\bm{x}_{j}$ in $\Omega_{i'}$, and}\quad
    \sum_{i,\bm{x}_j\in\partial\Omega_i}\mathcal{N}_{i;j} (u_i)=0.
\end{equation}
We should therefore split the Neumann contributions in such a way that
if properties~\eqref{eq:MonodomainDiscreteSolutionsProperty} are
satisfied for an iterate $u_i^{n}$, then the transmission conditions
do not change any more, $g_{i;j}^{n+1}=g_{i;j}^{n}$.  We show in the
next two sections that such a splitting can either be obtained using
auxiliary variables and communicating only with neighbors, or by
communicating with all subdomains that share the cross-point.

\section{Auxiliary variables at
  cross-points}\label{sect:AuxiliaryVariables}
We now show how to introduce auxiliary variables near the cross
points.  At the continuous level, we have on the interface between
subdomain $\Omega_i$ and $\Omega_{i'}$ from (\ref{TransmissionCond})
the identity
\begin{equation*}
  g^{n+1}_{i}=\frac{\partial u^{n+1}_i}{\partial\bm{n}_{ii'}}+pu^{n+1}_i
    =\frac{\partial u^{n}_{i'}}{\partial\bm{n}_{ii'}}+pu^{n}_{i'}
    =-\frac{\partial u^{n}_{i'}}{\partial\bm{n}_{i'i}}+pu^{n}_{i'}
    =-g^{n}_{i'}+2p u^n_{i'},
\end{equation*}
since by definition $g^{n}_{i'}=\frac{\partial
  u^{n}_{i'}}{\partial\bm{n}_{i'i}}+pu^{n}_{i'}$ and the normals are
in opposite directions. At the discrete level, the same equality can
be used to update the Robin transmission conditions,
\begin{equation}\label{eq:AlternateRobinUpdate}
  g^{n+1}_{i;j}=-g^{n}_{i';j}+2\frac{p}{2}u_{i';j}^{n}\sum_{j'}
    \lvert\bm{x}_{i;j}-\bm{x}_{i;j'}\rvert,
\end{equation}
where the sum is over all $j'$ such that
  $\lbrack\bm{x}_j\bm{x}_j'\rbrack$ is a boundary edge of $\mathcal{T}_i$.
This is very useful in practice, because one then does not even need
to implement a normal derivative evaluation~\cite{Gander:2001:OSH}.
At interface points which are not cross-points, this update will give
the same update as applying formula~\eqref{eq:RobinUpdate} using the
definition~\eqref{eq:NeumannDiscrete}. Therefore, if we are given the
values $g_{i,i';j}^n$ which represent the Robin transmission
information sent from subdomain $i'$ to subdomain $i$, we can compute
$u^{n+1}_i$ by setting 
\begin{equation}\label{eq:sumgiiprime}
g_{i;j}^{n+1}=\sum_{i'}g_{i,i';j}^{n}
\end{equation} 
and solving Eq~\eqref{eq:OSMStepSystem}. The sum in~\eqref{eq:sumgiiprime} above
is over all $i'$ such that there exists an edge originating from the vertex
$\bm{x}_j$ that belongs to both $\mathcal{T}_i$ and $\mathcal{T}_{i'}$.
 We then set
\begin{equation}\label{eq:RobinUpdateAuxiliary}
  g^{n+1}_{i',i;j}:=-g^{n}_{i,i';j}+2\frac{p}{2}u_{i;j}^{n+1}
    \sum_{j'}
    \lvert\bm{x}_{i;j}-\bm{x}_{i;j'}\rvert,
\end{equation}
where the sum is over all $j'$ such that
  $\lbrack\bm{x}_j\bm{x}_{j'}\rbrack$ is a boundary edge of both $\mathcal{T}_i$ and $\mathcal{T}_{i'}$.
For this we need however to store the auxiliary variables
$g^{n+1}_{i',i;j}$, because it is not possible to recover
$g^{n+1}_{i',i;j}$ from $u_i^{n+1}$ when $\bm{x}_j$ is a
cross-point. Only the sum over $i'$ of the $g^{n+1}_{i',i;j}$ can be
recovered from $u_i^{n+1}$. 

Since the $g^{n}_{i,i';j}$ represent a split of the
  discrete Robin conditions, we can deduce from them a split
  of the discrete Neumann conditions and introduce the
  $\mathcal{N}_{i',i;j}^{n}$. We set
\begin{equation}\label{eq:defN}
\mathcal{N}_{i,i';j}^{n+1}:=g^{n}_{i,i';j}-\frac{p}{2}\left(\sum_{j'}\lvert\bm{x}_j-\bm{x}_{j'}\rvert\right)u_{i;j}^{n+1},
\end{equation}
where the sum is over all $j'$ such that $\lbrack\bm{x}_{j}\bm{x}_{j'}\rbrack$ is a
boundary edge of both $\mathcal{T}_i$ and $\mathcal{T}_{i'}$. 
By Eqs.~\eqref{eq:LumpedBoundaryMatrix},~\eqref{eq:NeumannDiscrete}
and~\eqref{eq:OSMStepSystem}, we obtain
\begin{equation}\label{eq:sumNeumann}
\mathcal{N}_{i;j}(u_i^{n+1})=\sum_{i'}\mathcal{N}_{i,i';j}^{n+1},
\end{equation}
where the sum is over all $i'$ such that there exists an edge
originating from $\bm{x}_j$ that is a boundary edge of both
$\mathcal{T}_i$ and $\mathcal{T}_{i'}$.

\subsection{Convergence of the auxiliary variable method}

At the continuous level, one can prove convergence of OSM using energy
estimates, see for example~\cite{Lions:1990:SAM,
  Despres:1991:DomainDecompositionHelmholtzProblem}.  At the discrete
level, this technique fails in general
\cite{Gander.Kwok:ApplicabilityLionsEnergyEstimates}, precisely
because of the cross-points. 

We prove now convergence of OSM in the
presence of cross-points, when auxiliary variables are used.
\begin{lemma}\label{Lemma1}
  Let $\bm{f}=(f_j)$ be a right hand side of the discretized operator
  $\eta-\Lapl$ with $f_{i;j}$ such that $\sum_if_{i;j}=f_j$. Then there exist
  $g_{i,i';j}$ which are a fixed point of the discrete Optimized
  Schwarz algorithm with auxiliary variables near cross points.
\end{lemma}
\begin{proof}
Let $\bm{u}$ be the discrete mono-domain solution. 
Let $\bm{u}_i$ be the restriction of $\bm{u}$ to $\mathcal{T}_i$. Let
\begin{align*}
\mathcal{E}_{i;j}&:=\{j'',
               \lbrack\bm{x}_j\bm{x}_{j''}\rbrack\text{ boundary edge of
               $\mathcal{T}_i$}\},\\
\mathcal{E}_{i;i';j}&:=\{j'',
               \lbrack\bm{x}_j\bm{x}_{j''}\rbrack\text{ boundary edge of
               $\mathcal{T}_i$ and of $\mathcal{T}_{i'}$ }\}.
\end{align*}
We use formula~\eqref{eq:NeumannDiscrete} to obtain the existence of
$g_{i;j}$ such that the solution of~\eqref{eq:OSMStepSystem} are the
$\bm{u}_i$. For any given
cross-point node $\bm{x}_j$, we have to split the $g_{i;j}$ into
$g_{i,i';j}$ that satisfy
\begin{align*}
  g_{i;j}&=\sum_{i'\;\text{s.t.}\; \mathcal{E}_{i;i';j}\neq \emptyset}g_{i,i';j},\\
  g_{i',i;j}&=-g_{i,i';j}+2\frac{p}{2}u_{j}\sum_{j'\in \mathcal{E}_{i;i';j}}\lvert\bm{x}_{i,j}-\bm{x}_{i,j'}\rvert.
\end{align*}
Subtracting the Dirichlet parts on both sides in the first equation,
and transferring half the Dirichlet part in the second equation from
the right to the left,  we get 
\begin{align*}
  g_{i;j}-\frac{p}{2}u_{j}\sum_{j''\in\mathcal{E}_{i;j}}
    \lvert\bm{x}_{i,j}-\bm{x}_{i,j'}\rvert&=\sum_{i'\;\text{s.t.}\;\mathcal{E}_{i;i';j}\neq\emptyset} (g_{i,i';j}-\frac{p}{2}u_{j}
    \sum_{j''\in\mathcal{E}_{i,i';j}}
    \lvert\bm{x}_{i,j}-\bm{x}_{i,j'}\rvert),\\
  g_{i',i;j}-\frac{p}{2}u_{j}\sum_{j''\in\mathcal{E}_{i,i';j}}
    \lvert\bm{x}_{i,i',j}-\bm{x}_{i,j'}\rvert&=-(g_{i,i';j}-\frac{p}{2}u_{j}
    \sum_{j''\in\mathcal{E}_{i,i';j}}
    \lvert\bm{x}_{i,j}-\bm{x}_{i,j'}\rvert),
\end{align*}
We recognize the discrete Neumann conditions, see~\eqref{eq:defN}. So the
problem becomes the concrete splitting problem of Neumann conditions:
given $\mathcal{N}_{i;j}$, find $\mathcal{N}_{i,i';j}$ such that
\begin{align}\label{firsteqN}
\mathcal{N}_{i;j}&=\sum_{i'\;\text{s.t.}\;\mathcal{E}_{i;i';j}\neq\emptyset}\mathcal{N}_{i,i';j},\\
\mathcal{N}_{i,i';j}&=-\mathcal{N}_{i',i;j}.
\end{align}
By~\eqref{eq:MonodomainDiscreteSolutionsProperty}, since $\bm{u}$ is
the discrete mono-domain solution, we have
$\sum_i\mathcal{N}_{i;j}=0$. 
For each cross-point $\bm{x}_j$, we define a graph $G$, whose set of
vertices $V(G)$ and set of
edges $E(G)$ are defined as 
\begin{align*}
V(G)&=\{i,\;\bm{x}_j\in\Omega_i\},\\
E(G)&=\{\{i,i'\} \subset V(G),\;\text{$\mathcal{T}_i$ and
  $\mathcal{T}_{i'}$ share an edge originating from $\bm{x}_j$}\}.
\end{align*}
We apply now Lemma~\ref{lemma:graphes} to conclude the proof. 
\end{proof}

\begin{theorem}\label{ConvergenceTh}
The optimized Schwarz method (\ref{algo:ContinuousOSM}) discretized
with finite elements (\ref{FEMOSM}) and using auxiliary variables
for the transmission conditions is convergent.
\end{theorem}
\begin{proof}
Because of
Lemma~\ref{Lemma1}, we can assume without loss of generality that
$\bm{f}_i=0$. For each subdomain $\Omega_i$, 
we multiply the definition of the
discrete Neumann condition~\eqref{eq:NeumannDiscrete} by $u_{i;j}$, then sum over all
$j$ such that $\bm{x}_j$ belongs to $\overline{\Omega}_i$ to obtain
\begin{equation*}
\begin{split}
\int_{\Omega_i}\lvert\nabla u_i^{n+1}\rvert^2&+\eta\int_{\Omega_i}\lvert u_i^{n+1}\rvert^2\\
&=\sum_{\bm{x}_j\in\partial\Omega_i}\mathcal{N}_{i;j}^{n+1}u_{i;j}^{n+1}.\\
&=\sum_{i'}\sum_{\bm{x}_j\in\partial\Omega_i\cap\partial\Omega_{i'}}\mathcal{N}_{i,i';j}^{n+1}u_{i;j}^{n+1}\qquad\text{(by~\eqref{eq:sumNeumann})}\\
&=\sum_{i'}\sum_{\bm{x}_j\in\partial\Omega_i\cap\partial\Omega_{i'}}
\frac{\lvert\mathcal{N}_{i,i';j}^{n+1}+\frac{p}{2}\sum_{j''}\lvert\bm{x}_j-\bm{x}_{j''}\rvert
  u_{i;j}^{n+1}\rvert^2}
{2p \sum_{j''}\lvert\bm{x}_j-\bm{x}_{j''}\rvert }
\\&\phantom{=\sum_{i'}\sum_{\bm{x}_j\in\partial\Omega_i\cap\partial\Omega_{i'}}}-
\frac{\lvert\mathcal{N}_{i,i';j}^{n+1}-\frac{p}{2}\sum_{j''}\lvert\bm{x}_j-\bm{x}_{j''}\rvert
u_{i;j}^{n+1}\rvert^2}
{2p \sum_{j''}\lvert\bm{x}_j-\bm{x}_{j''}\rvert},\\
&=\sum_{i'}\sum_{\bm{x}_j\in\partial\Omega_i\cap\partial\Omega_{i'}}\frac{\lvert
  g_{i,i';j}^{n}\rvert^2-\lvert g_{i',i;j}^{n+1}\rvert^2}{2p
  \sum_{j''}\lvert\bm{x}_j-\bm{x}_{j''}\rvert }
\qquad\text{(by~\eqref{eq:defN} and~\eqref{eq:RobinUpdateAuxiliary})}.
\end{split}
\end{equation*}
We now sum over all subdomains $i$ and over the iteration index $n$ to get
\begin{equation*}
\begin{split}
\sum_{n=0}^{N}\sum_{i=1}^I\int_{\Omega_i}\lvert\nabla
u_i^{n+1}\rvert^2+\eta\int_{\Omega_i}\lvert u_i^{n+1}\rvert^2&=
\sum_{i,i'}\sum_{\bm{x}_j\in\partial\Omega_i\cap\partial\Omega_{i'}}\frac{\lvert
  g_{i,i';j}^{0}\rvert^2-\lvert g_{i',i;j}^{N+1}\rvert^2}{2p
  \sum_{j''}\lvert\bm{x}_j-\bm{x}_{j''}\rvert}\\
&\leq \sum_{i,i'}\sum_{\bm{x}_j\in\partial\Omega_i\cap\partial\Omega_{i'}}\frac{\lvert
  g_{i,i';j}^{0}\rvert^2}{2p
  \sum_{j''}\lvert\bm{x}_j-\bm{x}_{j''}\rvert}.
\end{split}
\end{equation*}
This shows that the sum over the energy over all iterates and
subdomains stays bounded, as the iteration number $N$ goes to
infinity, which implies that the energy of the iterates, and hence the
iterates converge to zero.  \qedhere
\end{proof}

\subsection{Numerical observation using auxiliary variables} 
  \label{subsect:NumericalObservationsAuxiliaryVariables}

Using auxiliary variables can have surprising numerical side
effects. We show in Figure~\ref{fig:ErrorsEpsilonMachineUnderflowLevel}
 \begin{figure}
  \centering
  \includegraphics[width=0.7\textwidth]{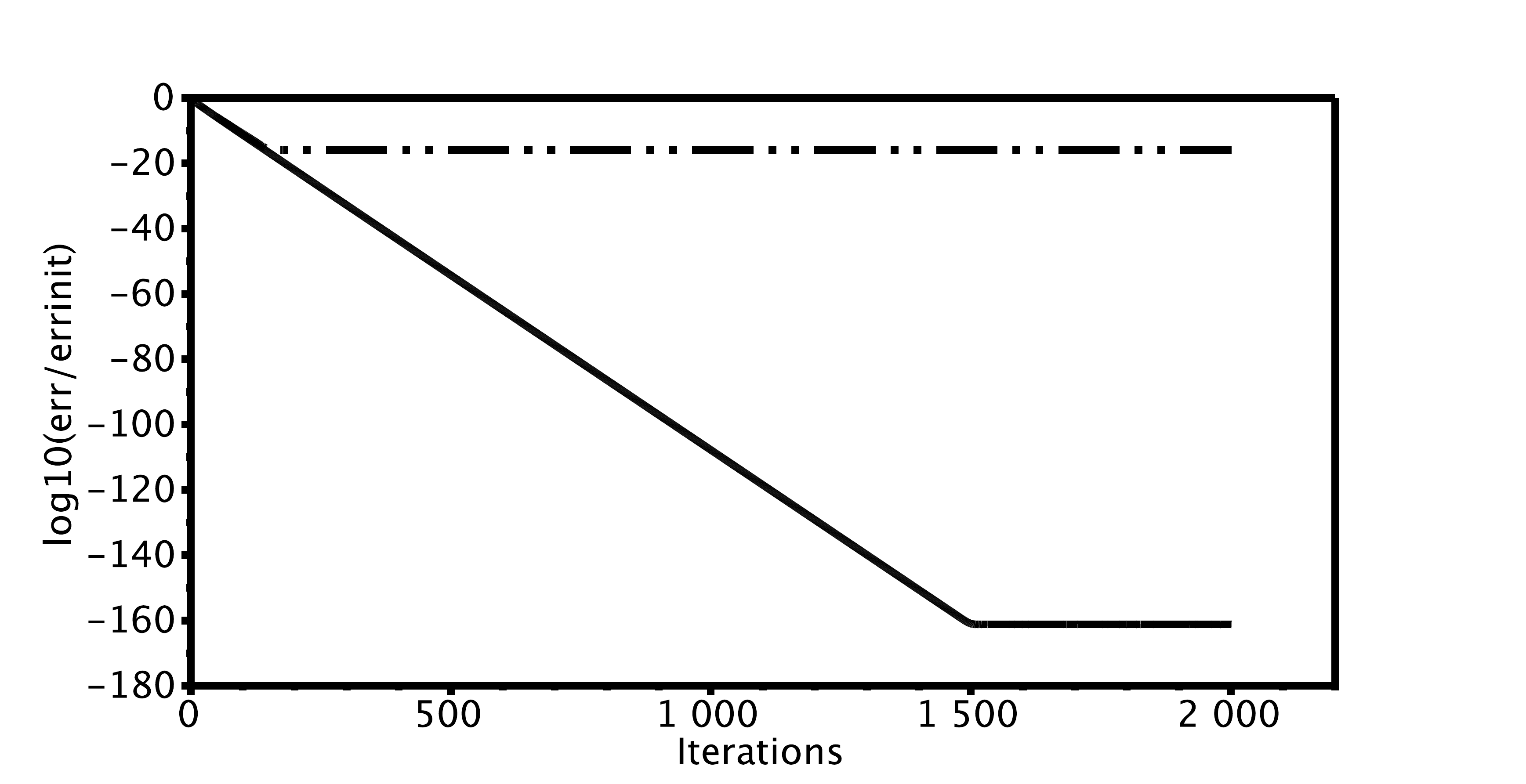}
  \caption{Error using OSM with auxiliary variables for $4\times1$
    (solid) and $2\times2$ (dashed-dotted) subdomains}
  \label{fig:ErrorsEpsilonMachineUnderflowLevel}
\end{figure}
the error measured in $L^\infty$ of OSM with auxiliary variables for
the domain $\Omega=(0,4)^2$ decomposed once into $2\times 2$
subdomains and once into $4\times1$ subdomains, for $p=2.0$ and
$\eta=0.0$ and mesh size $h=1/10$.  We iterate directly on the error
equations, $f=0$, and initialize the transmission conditions with
random values. We observe that in the presence of cross points,
convergence stagnates around the machine precision, whereas without,
the stagnation comes much later.

To understand these results, we need to consider floating point
arithmetic, see~\cite{IEEE-754:2008, Sterbenz:1973:Floating,
  Goldberg:1991:WhatEvery}, and in particular the machine precision 
{\tt macheps} and the smallest positive floating point number {\tt
  minreal}.  We used in the above experiment double precision in
\texttt{C++} so {\tt macheps}$=2^{-53}\approx1.1\cdot^{-16}$ and
{\tt minreal}$\approx 4.9\cdot 10^{-324}$.  Had we been computing a
real problem with nonzero right hand side $f$, we would expect
stagnation near the machine precision. However, when iterating
directly on the errors, stagnation should occur much later, at the
level of the smallest positive floating point number.  

To analyze the
early stagnation observed, we consider a simple model problem with
$2\times 2$ subdomains, see Figure~\ref{fig:DegenerateCase},
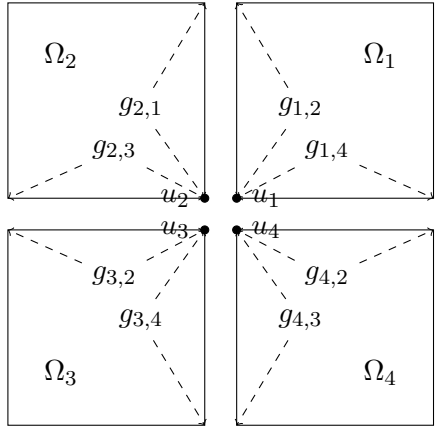
\begin{figure}
\centering
\begin{tikzpicture}[scale=0.7]

\draw (-4,-4) rectangle (-0.3,-0.3);
\draw (-4,0.3) rectangle (-0.3,4);
\draw (0.3,-4) rectangle (4,-0.3);
\draw (0.3,0.3) rectangle (4,4);

\node at (3,3) {$\Omega_1$};
\node at (-3,3) {$\Omega_2$};
\node at (-3,-3) {$\Omega_3$};
\node at (3,-3) {$\Omega_4$};

 \node[vertex,label=right:{$u_1$}] at (0.3,0.3) {}; 
\node[vertex, label=left:{$u_2$}] at (-0.3,0.3) {}; 
\node[vertex, label=left:{$u_3$}] at (-0.3,-0.3) {}; 
\node[vertex, label=right:{$u_4$}] at (0.3,-0.3) {}; 

\node (A12) at (1.5,2) {$g_{1,2}$}; 
\draw[dashed,->] (A12) -- (0.3,0.3);
\draw[dashed,->] (A12) -- (0.3,4);

\node (A14) at (2,1.2) {$g_{1,4}$}; 
\draw[dashed,->] (A14) -- (0.3,0.3);
\draw[dashed,->] (A14) -- (4,0.3);

\node (A21) at (-1.5,2) {$g_{2,1}$}; 
\draw[dashed,->] (A21) -- (-0.3,0.3);
\draw[dashed,->] (A21) -- (-0.3,4);

\node (A23) at (-2,1.2) {$g_{2,3}$}; 
\draw[dashed,->] (A23) -- (-0.3,0.3);
\draw[dashed,->] (A23) -- (-4,0.3);

\node (A34) at (-1.5,-2) {$g_{3,4}$}; 
\draw[dashed,->] (A34) -- (-0.3,-0.3);
\draw[dashed,->] (A34) -- (-0.3,-4);

\node (A32) at (-2,-1.2) {$g_{3,2}$}; 
\draw[dashed,->] (A32) -- (-0.3,-0.3);
\draw[dashed,->] (A32) -- (-4,-0.3);

\node (A43) at (1.5,-2) {$g_{4,3}$}; 
\draw[dashed,->] (A43) -- (0.3,-0.3);
\draw[dashed,->] (A43) -- (0.3,-4);

\node (A41) at (2,-1.2) {$g_{4,2}$}; 
\draw[dashed,->] (A41) -- (0.3,-0.3);
\draw[dashed,->] (A41) -- (4,-0.3);
\end{tikzpicture}
\caption{Degenerate case.}\label{fig:DegenerateCase}
\end{figure}
where there is exactly one $Q_1$ element per subdomain and the only
interior node is a cross-point.  This means the mono-domain solutions
$u$ is a scalar. We thus have $\Omega=(-h,h)\times(-h,h)$, and for the
subdomains $\Omega_1=(0,h)\times(0,h)$, $\Omega_2=(-h,0)\times (0,h)$,
$\Omega_3=(-h,0)\times(-h,0)$,
$\Omega_4=(0,h)\times(-h,0)$. We apply the OSM with lumped
 Robin transmission conditions and $f=0$. Since
there is only one interior node in the whole mesh, there is only
a single test function $\phi$ with
$\phi(x,y)= (1-\lvert x\rvert) (1-\lvert y\rvert)$.
By Eq. \eqref{eq:OSMStepSystem}, we have
\begin{equation*}
\begin{split} 
A_{1}=A_{2}=A_{3}=A_{4}&=\eta h^2\left(\int_0^1(1-x)^2\ud x\right)^2+\int_0^1(1-x)^2\ud x+\int_0^1(1-y)^2\ud y\\
&=\frac{\eta h^2}{9}+\frac{2}{3}.
\end{split}
\end{equation*} 
We use lumped Robin transmission conditions, and
by~\eqref{eq:LumpedBoundaryMatrix}, we get
\begin{equation*}
B_{1}=B_{2}=B_{3}=B_{4}= \frac{p}{2}h\left(\int_0^1(1-x)\ud x+\int_0^1(1-y)\ud y\right)=ph.
\end{equation*}
Therefore, we have
by~\eqref{eq:OSMStepSystem} and~\eqref{eq:sumgiiprime}
\begin{equation*}
u_{i}^{n+1}=\frac{g_i^{n+1}}{\frac{2}{3}+\frac{\eta h^2}{9}+ph},\quad\text{$i=1,\ldots,4$}.
\end{equation*}
Thus, for the OSM iteration, we obtain
\begin{equation}\label{eq:ufromginDegenerate}
\begin{aligned}
u_{1}^{n+1}&=\frac{g_{12}^n+g_{14}^n}{\frac{2}{3}+\frac{\eta h^2}{9}+ph},&
 u_{2}^{n+1}&=\frac{g_{23}^n+g_{21}^n}{\frac{2}{3}+\frac{\eta h^2}{9}+ph},\\
u_{3}^{n+1}&=\frac{g_{32}^n+g_{34}^n}{\frac{2}{3}+\frac{\eta h^2}{9}+ph},&
u_{4}^{n+1}&=\frac{g_{43}^n+g_{41}^n}{\frac{2}{3}+\frac{\eta h^2}{9}+ph},
\end{aligned}
\end{equation}
and by~\eqref{eq:RobinUpdateAuxiliary}, we get
\begin{equation*}
  g^{n+1}_{i',i}:=-g^{n}_{i,i'}+phu_{i}^{n+1}.
\end{equation*}
Eliminating the $u_i^{n+1}$ from the iteration leads to
\begin{equation*}
\arraycolsep0.25em
\begin{bmatrix}
g_{1,2}^{n+1}\\g_{2,1}^{n+1}\\g_{2,3}^{n+1}\\g_{3,2}^{n+1}\\g_{3,4}^{n+1}\\g_{4,3}^{n+1}\\g_{4,1}^{n+1}\\g_{1,4}^{n+1}
\end{bmatrix}
=
\begin{bmatrix}
0                & \alpha-1 & \alpha         &0                &0                 &0                &0               &0\\
\alpha-1 &0                 &          0        &0                &0                 &0                &0               &\alpha\\
0                &0                 &          0        &\alpha-1 &\alpha         &0                &0               &0\\
0                & \alpha        & \alpha-1  &0                &0                 &0                &0               &0\\
0                &0                 &0                  &0                &         0        &\alpha-1 &\alpha       &0\\
0                &0                 &0                  & \alpha       &\alpha-1  &0                &0               &0\\
\alpha        &0                 &0                  &0                 &0                &0                &0               &\alpha-1 \\
0                &0                 &0                  &0                 &0                &\alpha        &\alpha-1 &0
\end{bmatrix}
\begin{bmatrix}
g_{1,2}^{n}\\g_{2,1}^{n}\\g_{2,3}^{n}\\g_{3,2}^{n}\\g_{3,4}^{n}\\g_{4,3}^{n}\\g_{4,1}^{n}\\g_{1,4}^{n},
 \end{bmatrix},
\end{equation*}
where we introduced the scalar quantity
\begin{equation*}
\alpha=\frac{ph}{\frac{\eta h^2}{9}+\frac{2}{3}+ph}.
\end{equation*}
Since $0<\alpha<1$, the $\ell^\infty$ norm of this iteration matrix is
$1$, and hence its spectral radius is bounded by $1$.  Note however
that $1$ and $-1$ are eigenvalues of this matrix, with corresponding
eigenvectors
\[
(-1,1,-1,1,-1,1,-1,1)^T
\quad\mbox{and}\quad
(1,1,-1,-1,1,1,-1,-1)^T.
\]
This shows that the vector of auxiliary variables will not converge to
$0$ in general. 
However, the modes with eigenvalue $+1$ and $-1$ make
no contribution to the $u_i$, see Eq.~\eqref{eq:ufromginDegenerate}, so the algorithm will
converge for the $u_i^n$, as proved in Theorem~\ref{ConvergenceTh}. In
floating point arithmetic however, the fact that the auxiliary
variables do not converge (and remain $O(1)$ because of their
initialization) prevents the algorithm applied to the error equations
to converge in $u_i^n$ below the machine precision, as we observed in
Figure~\ref{fig:ErrorsEpsilonMachineUnderflowLevel}. Luckily, this has
no influence when solving a real problem with non-zero right hand
side, but must be remembered when testing codes.

\section{Complete communication method}\label{sect:CompleteCommunications}

We now present a different approach, not using auxiliary variables,
but still guaranteeing that the discrete mono-domain solution is a
fixed point of the discrete OSM. This requires subdomains to
communicate at cross-points with every subdomain sharing the
cross-point. Most methods obtained algebraically using matrix
splittings use complete communication. To get Domain Decomposition
methods directly from the matrix, one usually duplicates the
components corresponding to the nodes lying on the interfaces between
subdomains so that each node is present in the matrix as many times as
the number of subdomains it belongs to, see for
example~\cite{Gander.Kwok:2012:BestRobinParameters,loisel2013condition}. To
prove convergence of this approach needs however different techniques
from the energy estimates, see
\cite{Gander.Kwok:2012:BestRobinParameters,loisel2013condition}. 

\subsection{Keeping the discrete mono-domain solution a fixed point}\label{subsect:MonodomainFixedPoint}

Consider a cross point $\bm{x}_j$ belonging to subdomains
$\overline{\Omega}_{i}$ for $i$ in $\{1,\ldots,I\}$ with $I\geq 3$. We
consider local linear updates for the discrete Robin transmission
conditions at cross-points of the form
\begin{equation*}
  g_{i;j}^{n+1}=\ell_{\mathcal{D}}((u_{i;j}^{n})_{1\leq i\leq I})
    +\ell_{\mathcal{N}}((\mathcal{N}_{i;j}(u_{i})))_{1\leq i\leq I}),
\end{equation*}
where $\ell_{\mathcal{D}}$ and $\ell_{\mathcal{N}}$ are linear
maps from $\mathbb{R}^I$ to $\mathbb{R}^I$, which can be
represented by matrices,
\begin{equation}
\begin{bmatrix}
g_{1;j}^{n+1}\\\vdots\\\vdots\\g_{I;j}^{n+1}
\end{bmatrix}
=\mathbf{A}_{\mathcal{D}}
\begin{bmatrix}
u_{1;j}^{n}\\\vdots\\\vdots\\u_{I;j}^{n}
\end{bmatrix}
+
\mathbf{A}_{\mathcal{N}}
\begin{bmatrix}
\mathcal{N}_{1;j}^{n}\\\vdots\\\vdots\\\mathcal{N}_{I;j}^{n}
\end{bmatrix}.
\end{equation}
At the cross point $\bm{x}_j$, the mono-domain
solution satisfies \eqref{eq:MonodomainDiscreteSolutionsProperty}, i.e.
\begin{equation}\label{eq:MonodomainCrossPoints}
  u_{i;j}=u_{1;j}\text{ for all $i$ in $\{1,\ldots,I\}$},\quad
    \sum_{i=1}^I\mathcal{N}_{i;j}(u_{i})=0.
\end{equation}
For the mono-domain solution to be a fixed point, $g^{n+1}_{i;j}$
should be equal to $g^{n}_{i;j}$ whenever
conditions~\eqref{eq:MonodomainCrossPoints} are satisfied.  Therefore,
the matrices must satisfy
\begin{align}\label{eq:FixedPointConditions}
  (\mathbf{A}_{\mathcal{N}})_{ii'}&=\delta_{i,i'}-\alpha_{i},&
    \sum_{i'=1}^I(\mathbf{A}_{\mathcal{D}})_{ii'}
    &=\frac{p}{2}\sum_{\begin{subarray}{l}j''\;\text{s.t.}\;
               \lbrack\bm{x}_j\bm{x}_{j''}\rbrack\text{ is a}\\\text{ boundary edge of }\mathcal{T}_i \end{subarray}}\lvert\bm{x}_j-\bm{x}_{j''}\rvert,
\end{align}
for some constants $\alpha_i$.

\subsection{An intuitive Neumann splitting near cross-points}\label{subsect:SplittingNeumann}

\begin{figure}
\centering
\begin{tikzpicture}[scale=1.5]
\node[vertex] (O) at (0,0.5) {} ;
\coordinate (A) at (50:2) {};
\coordinate (B) at (110:2.4) {};
\coordinate (C) at (190:1.8) {};
\coordinate (D) at (250:2.2) {};
\coordinate (E) at (340:2.25) {};
\draw[double,thick] (O) -- (barycentric cs:O=-0.25,A=0.75);
\draw[double,thick] (O) -- (barycentric cs:O=-0.1,B=0.75);
\draw[double,thick] (O) -- (barycentric cs:O=-0.25,C=0.75);
\draw[double,thick] (O) -- (barycentric cs:O=-0.1,D=0.75);
\draw[double,thick] (O) -- (barycentric cs:O=-0.25,E=0.75);

\node at (barycentric cs:O=-0.2,A=0.5,B=0.5) {$\Omega_1$};
\node at (barycentric cs:O=-0.2,B=0.5,C=0.5) {$\Omega_2$};
\node at (barycentric cs:O=-0.2,C=0.5,D=0.5) {$\Omega_3$};
\node at (barycentric cs:O=-0.2,D=0.5,E=0.5) {$\Omega_4$};
\node at (barycentric cs:O=-0.2,E=0.5,A=0.5) {$\Omega_5$};

\draw (A) -- (B) -- (C) -- (D) -- (E) -- (A) ;

\draw[dashed] (O) --  (80:1.5) -- (B);
\node at (80:1.5) {$\mathcal{N}_1^+$};

\draw[dashed] (O) --  (130:1) -- (B);
\node at (130:1) {$\mathcal{N}_2^-$};

\draw[dashed] (O) --  (270:1 ) -- (D);
\node at (270:1) {$\mathcal{N}_4^-$};

\draw[dashed] (O) --  (320:1 ) -- (E);
\node at (320:1) {$\mathcal{N}_4^+$};

\end{tikzpicture}
\caption{Splitting of $\mathcal{N}_i$ into $\mathcal{N}_i^+$ and $\mathcal{N}_i^-$}
\label{fig:DecompositionNeumann}
\end{figure}
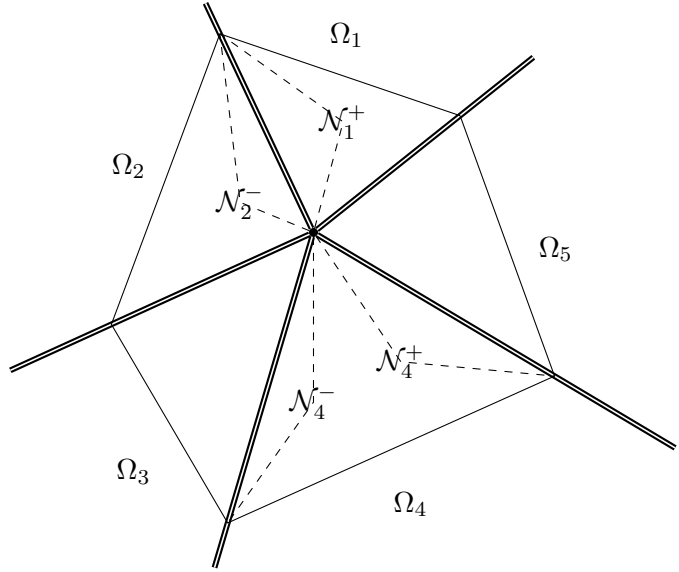

Suppose we are given $I$ values $(\mathcal{N}_i)_{i=1,\ldots,I}$, each
representing the discrete Neumann values at $\bm{x}_j$ for subdomain
$\Omega_i$. Our goal is to find a splitting
$(\mathcal{N}_i^+,\mathcal{N}_i^-)_{i=1,\ldots,I}$ such that
\begin{equation}
  \mathcal{N}_i=\mathcal{N}_i^++\mathcal{N}_i^-.
\end{equation}
There are obviously many such splittings. At the
continuous level, the mono-domain solution has no Neumann jumps at the
interface between subdomains. It thus makes sense, at an intuitive level, to search for a
splitting minimizing the Neumann jumps
$\mathcal{N}_{i+1}^-+\mathcal{N}_i^+$, see Fig.~\ref{fig:DecompositionNeumann}.
Therefore, we choose to minimize
\begin{equation*}
  \sum_{i=1}^I\lvert \mathcal{N}_{i}^{+}+\mathcal{N}_{i+1}^{-}\rvert^2,
\end{equation*}
where, by convention,  $\mathcal{N}_{I+1}^{-}$ denotes $\mathcal{N}_{1}^{-}$.
We will see that this still does not give a unique solution,
but all such splittings give rise to the same transmission
conditions in the OSM discretized by finite elements.

We denote by $\bm{a}\in\mathbb{R}^I$ the vector with
$a_i=\mathcal{N}_i^-$, which implies
$\mathcal{N}_i^+=\mathcal{N}_i-a_i$. We thus search for $\bm{a}$ in
$\mathbb{R}^I$ such that the function
\begin{equation*}
  \bm{a}\mapsto\sum_{i=1}^I\lvert -a_{i}+\mathcal{N}_{i}+a_{i+1} \rvert^2
\end{equation*}
is minimized, i.e. we want to compute the solution of
\begin{equation}\label{eq:LeastSquaredProblem}
  \mathop{\mathrm{argmin}}_{\bm{a}\in\mathbb{R}^I}
    \lVert\mathbf{L}\bm{a}-\bm{\mathcal{N}}\rVert_2^2,
\end{equation}
where the matrix $\mathbf{L}=(\ell_{ii'})_{1\leq i,i'\leq I}$ with
\begin{equation*}
\ell_{ij}=\begin{cases}
1&\text{if $i'=i$,}\\
-1&\text{if $i'=i+1\mod I$},\\
0&\text{otherwise},
\end{cases}
\end{equation*}
or more explicitly
\begin{equation*}
\mathbf{L}=
\begin{bmatrix}
1&-1&0&\ldots&0&0\\
0&1&-1&\ddots&\ddots&0\\
\vdots&\ddots&\ddots&\ddots&\ddots&\vdots\\
\vdots&\ddots&\ddots&\ddots&\ddots&\vdots\\
0&\ddots&\ddots&0&1&-1\\
-1&0&\ldots&0&0&1
\end{bmatrix}.
\end{equation*}
Equation~\eqref{eq:LeastSquaredProblem} is a standard least squared
problem, but its solution is not unique, since
$\Ker(\mathbf{L})=\mathbb{R}\lbrack1,\ldots,1\rbrack^{T}$.  If
we require in addition that $\bm{a}$ is orthogonal to $\Ker(\mathbf{L})$,
then $\bm{a}$ is unique and
\begin{equation}
  \bm{a}=\mathbf{L}^{\dagger}\bm{\mathcal{N}},
\end{equation}
where $\mathbf{L}^{\dagger}$ is the pseudo-inverse of $\mathbf{L}$,
and all the solutions to~\eqref{eq:LeastSquaredProblem} are then of
the form $\mathbf{L}^{\dagger}\bm{\mathcal{N}}
+\mathbb{R}\lbrack1,\ldots,1\rbrack^{T}$.

Since $\mathbf{L}$ is a circulant matrix, its pseudo-inverse
$\mathbf{L}^{\dagger}$ is also a circulant matrix. Let
$(\mu_i)_{i\in\mathbb{Z}}$ be $I$-periodic such that
$\ell^{\dagger}_{ii'}=\mu_{i'-i}$, which implies
\begin{equation*}
\mathbf{L}^{\dagger}=
\begin{bmatrix}
\mu_0&\mu_1&\cdots&\mu_{I-1}\\
\mu_{I-1}&\mu_0&\ddots&\vdots\\
\vdots&\ddots&\ddots&\mu_1\\
\mu_{1}&\cdots&\mu_{I-1}&\mu_{0}
\end{bmatrix}.
\end{equation*}
In addition, since
$\Ker(\mathbf{L})=\mathbb{R}\lbrack1,\ldots,1\rbrack^{T}$,
we have
\begin{equation*}
\mathbf{L}^{\dagger}\mathbf{L}
=\mathbf{I}-\frac{1}{I}\begin{bmatrix}
1&\ldots&1\\
\vdots&\ddots&\vdots\\
1&\ldots&1
\end{bmatrix},
\end{equation*}
and therefore,
$$
\mu_0-\mu_{I-1}=1-\frac{1}{I}\quad\mbox{and}\quad
\mu_{i}-\mu_{i-1}=-\frac{1}{I}\;\text{for all $1\leq i\leq I$}.
$$
Therefore, for all $i=0,\ldots,I-1$ we get
\begin{equation*}
  \mu_i=\mu_0-\frac{i}{I}.
\end{equation*}
Moreover, $\mathop{\mathrm{range}}(\mathbf{L}^{\dagger})=\Ker(\mathbf{L})^{\bot}$, and
therefore $\sum_{i=0}^{I-1}\mu_i=0$, which yields
$\mu_0=\frac{I-1}{2}$. Therefore, for all $i=0,\ldots,I-1$,
\begin{equation*}
  \mu_i=\frac{I-1}{2}-\frac{i}{I}.
\end{equation*}
We thus obtain for the solution of the least squares problem
\begin{equation*}
\begin{split}
  a_i&=\sum_{i'=1}^{I}\mu_{i'-i}\mathcal{N}_{i'},
\end{split}
\end{equation*}
which gives for the splitting of the Neumann values
\begin{align*}
\mathcal{N}_i^+&=\sum_{i'=1}^{I}\mu_{i'-i}\mathcal{N}_{i'},&
\mathcal{N}_i^-&=\mathcal{N}_i-\sum_{i'=1}^{I}\mu_{i'-i}\mathcal{N}_{i'}.
\end{align*}
 We can use this splitting now in the OSM to exchange the Neumann
contributions $\mathcal{N}_i^+$ and $\mathcal{N}_{i+1}^-$ in the Robin
transmission conditions, \textit{i.e.}, we set
\begin{equation*}
\begin{split}
 (\mathbf{A}_{\mathcal{N}}\bm{\mathcal{N}})_i &= -\mathcal{N}_{i+1}^{-}-\mathcal{N}_{i-1}^{+},\\
&=-\mathcal{N}_{i-1}+\sum_{i'=1}^{I}\mu_{i'-i+1}\mathcal{N}_{i'}-\sum_{i'=1}^{I}\mu_{i'-i-1}\mathcal{N}_{i'},\\
&=-\mathcal{N}_{i-1}+\sum_{i'=1}^{I}(\mu_{i'-i+1}-\mu_{i'-i-1})\mathcal{N}_{i'}.
\end{split}
\end{equation*}
But
\begin{equation*}
\mu_{i'-i+1}-\mu_{i'-i-1}=
\begin{cases}
1-\frac{2}{I}&\text{if $i'=i\mod I$},\\
1-\frac{2}{I}&\text{if $i'=i-1\mod I$},\\
-\frac{2}{I}&\text{otherwise}.
\end{cases}
\end{equation*}
Therefore, we set
\begin{equation*}
(\mathbf{A}_{\mathcal{N}}\bm{\mathcal{N}})_i=\mathcal{N}_i-\frac{2}{I}\sum_{i'=1}^{I}\mathcal{N}_{i'}.
\end{equation*}

\subsection{An intuitive splitting of the Dirichlet part}\label{subsect:SplittingDirichlet}
We must choose a matrix $\mathbf{A}_{\mathcal{D}}$
satisfying~\eqref{eq:FixedPointConditions}, \textit{i.e.}, satisfy:
\begin{equation*}
 \sum_{i'=1}^I(\mathbf{A}_{\mathcal{D}})_{ii'}
    =\frac{p}{2}\sum_{\begin{subarray}{l}j'',\bm{x}_{j''}\in\partial\Omega_i,\\
               \lbrack\bm{x}_j\bm{x}_{j''}\rbrack\text{ edge of $\mathcal{T}_i$} \end{subarray}}\lvert\bm{x}_j-\bm{x}_{j''}\rvert,
\end{equation*}
There are also many possible choices for
$(\mathbf{A}_{\mathcal{D}})_{ii'}$, but in contrast to the Neumann
conditions which are only known variationally, the Dirichlet values
are known on the boundary. 
Therefore, to split 
the sum of
$\lvert\bm{x}_{j}-\bm{x}_{j''}\rvert$, we look at which neighbouring
subdomain the edge $\lbrack\bm{x}_j\bm{x}_{j''}\rbrack$ belongs to: if one is $\overline{\Omega}_i$, and the other is
$\overline{\Omega}_{i'}$, then we put $p\lvert\bm{x}_j-\bm{x}_{j''}\rvert$  into
$(\mathbf{A}_{\mathcal{D}})_{ii'}$. Hence, we set
\begin{equation*}
(\mathbf{A}_{\mathcal{D}})_{ii'}
    =\begin{cases}\frac{p}{2}\sum_{\begin{subarray}{l}j'',\bm{x}_{j''}\in\partial\Omega_i\cap\partial\Omega_{i'},\\
               \lbrack\bm{x}_j\bm{x}_{j''}\rbrack\text{ edge of
               }\mathcal{T}_i \end{subarray}}\lvert\bm{x}_j-\bm{x}_{j''}\rvert&\text{if
           $i'\neq i$,}\\
         0&\text{if $i'=i$.}
           \end{cases}
\end{equation*}

\subsection{Numerical simulations}\label{subsect:CompleteNumericalSimulations}
\begin{figure}
\centering
  \includegraphics[width=0.7\textwidth]{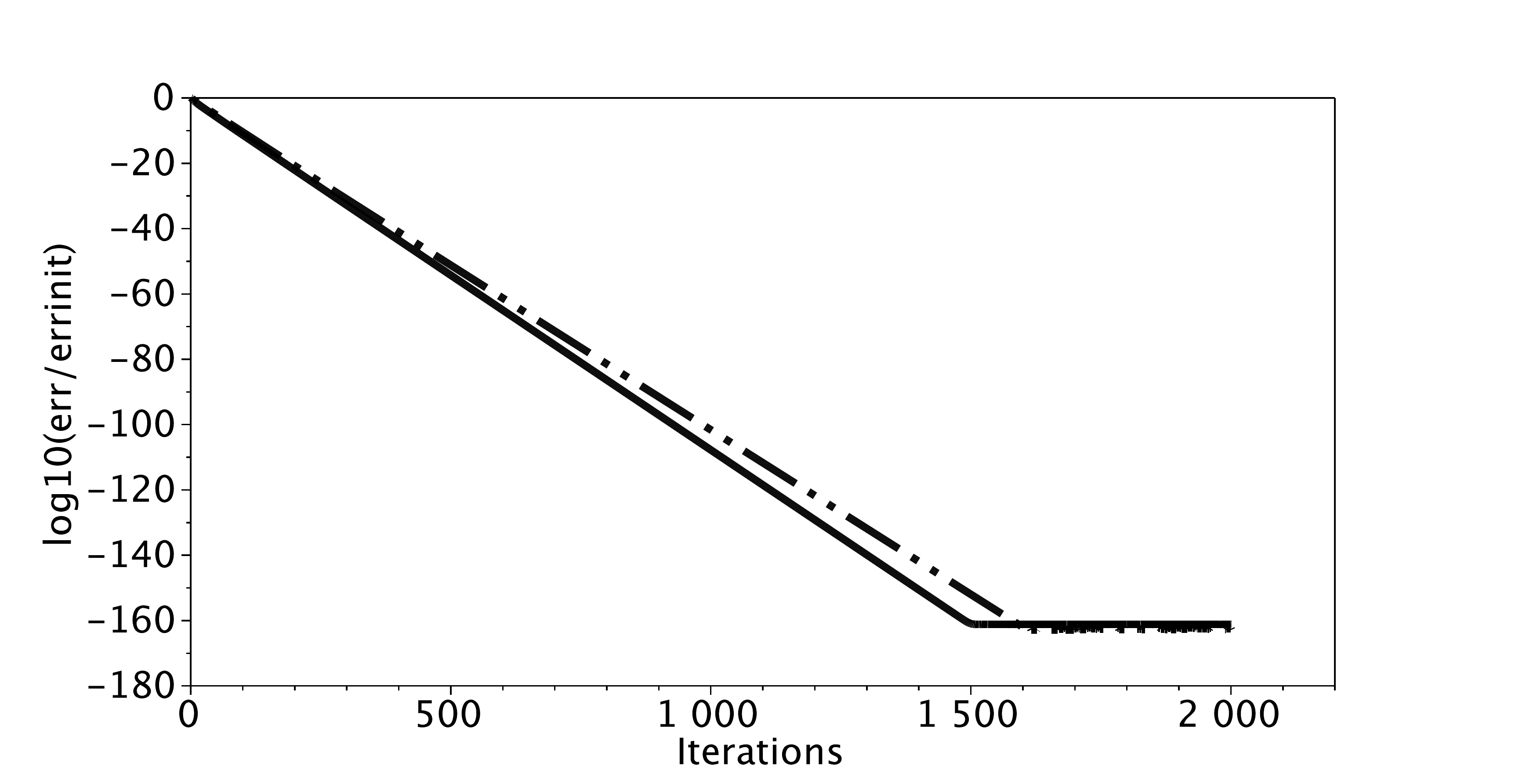}
\caption{Numerical convergence of the complete
    communications method  for $4\times1$  (solid)
    and $2\times2$ (dashed-dotted) subdomains}
\label{fig:CompleteCommunicationsNumericalSimulation}
\end{figure}

We do the same experiment for the complete communication method  as we did
for the auxiliary variable
in~\S\ref{subsect:NumericalObservationsAuxiliaryVariables}.
The results  are shown in
Figure~\ref{fig:CompleteCommunicationsNumericalSimulation}. 
As expected, for the complete communication method,
convergence is also observed up to \texttt{minreal} for the $2\times2$
subdomain cases, \textit{i.e.}, when there are crosspoints. 
In practice, when using complete communication methods, the Robin
parameters should be different at cross-points,
see~\cite{Gander.Kwok:2012:BestRobinParameters} for full details. In
this paper, we chose not to do so and use the same $p$ at
cross-points. 

\section{Conclusion}\label{sect:Conclusions}

This paper contains two concrete propositions on how to discretize
Neumann conditions at cross points in domain decomposition methods:
the auxiliary variable method and complete communication. We showed
three new results: first that the introduction of auxiliary variables
makes it possible to prove convergence of the discretized methods for
very general decompositions, including cross points, using energy
estimates. Second that Neumann conditions can be split at cross points
in a way minimizing artificial oscillation in the domain
decomposition, and third, in the Appendix, that lumping the mass matrix in a finite
element discretized optimized Schwarz method leads to better
performance. We explained this by a reinterpretation at the continuous
level, which shows a tangential higher order operator appearing.  Its
weight can even be optimized using the new concept of overlumping, and
this can be done purely at the algebraic level, without need to
discretize a complicated higher order operator.

We have restricted ourselves to two spatial dimensions. In higher
dimensions, in addition to cross-points, there would also be
cross-edges. Both the auxiliary variables method and complete
communication can be adapted to higher dimensions, which is work in
progress.

\appendix
\section{(Over)lumping of the Interface mass matrix}\label{sect:Lumping}

We start with a numerical experiment, using the consistent interface
mass matrix $B_i$ from~\eqref{eq:UnlumpedBoundaryMatrix} and the
lumped interface mass matrix $B_i^{\textrm{lump}}$
from~\eqref{eq:LumpedBoundaryMatrix} in the Robin transmission
condition of the OSM. We solve the Poisson equation with right hand
side $f(x,y)=2(y(4.0-y)+x(4.0-x))$ on the square domain
$\Omega=(0,4)^2$ with $3\times 3$ subdomains of equal size, and Robin
parameter $p=2.0$, discretized using $Q_1$ finite elements with mesh
size $h=1/15$. Figure~\ref{fig:CompLumpedNotLumped} shows how the
error decreases as a function of the iteration index in the OSM for
these two choices.
\begin{figure}
  \centering
  \includegraphics[width=0.8\textwidth]{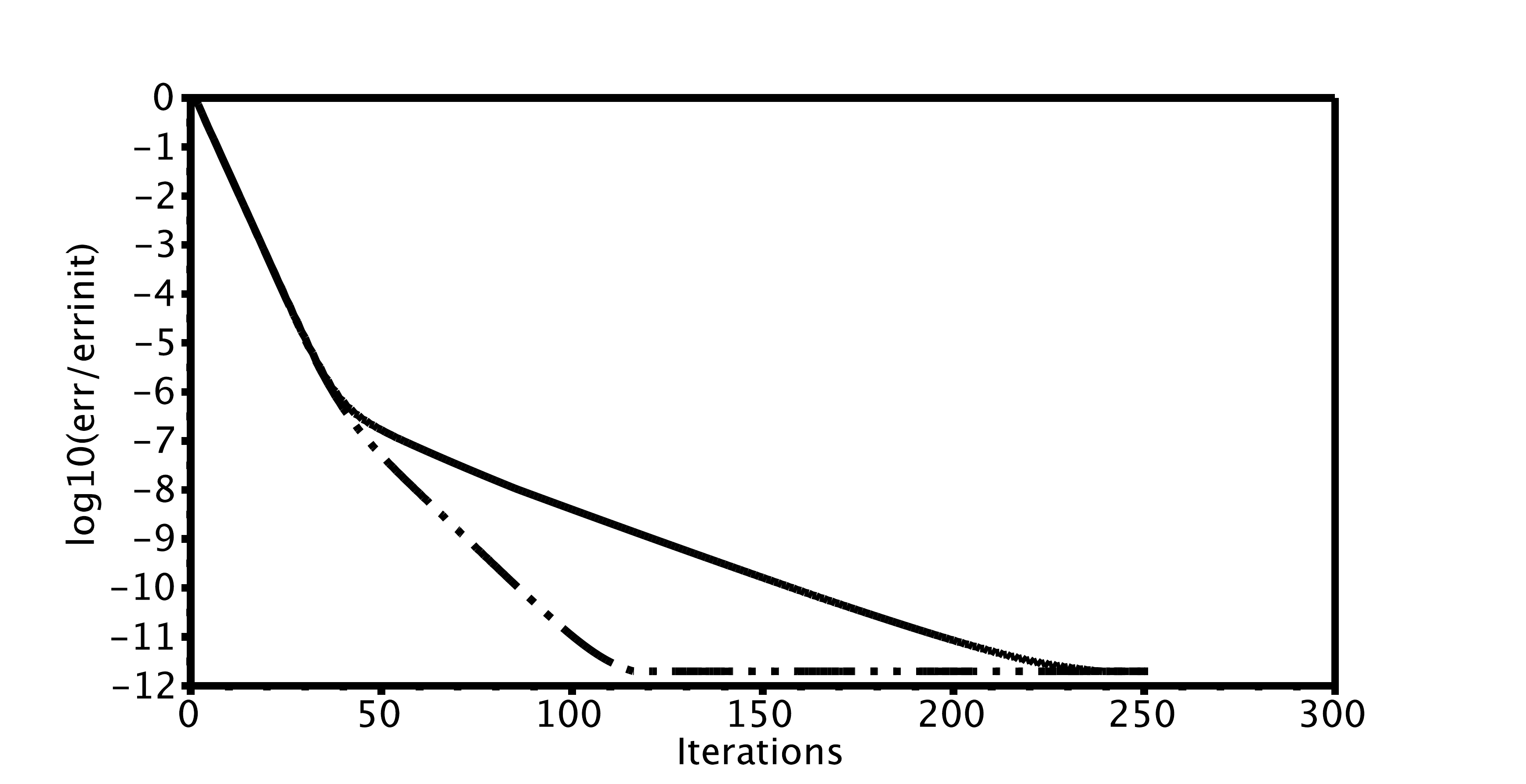}
  \caption{Convergence with lumped Robin(dashed-dotted) and consistent Robin(solid)
    }
  \label{fig:CompLumpedNotLumped}
\end{figure}
We see that initially the two methods converge at the same rate, but
around iteration $40$, the method using the consistent mass interface
matrix slows down. We show in Figure~\ref{fig:WhyLumpIsNecessary}
snapshots of the error distribution for selected iteration indices.
\begin{figure}
  \centering
  \includegraphics[width=0.3\textwidth]{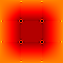}\quad 
  \includegraphics[width=0.3\textwidth]{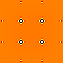}\quad 
  \includegraphics[width=0.3\textwidth]{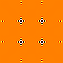}\\[1em]
  \includegraphics[width=0.3\textwidth]{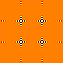}\quad 
  \includegraphics[width=0.3\textwidth]{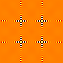}\quad 
  \includegraphics[width=0.3\textwidth]{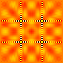}
  \caption{Scaled error distribution at iteration $35$, $50$, $75$,
  $100$, $150$ and $200$ for OSM with consistent interface mass matrix
  using auxiliary variables at cross-points.}
  \label{fig:WhyLumpIsNecessary}
\end{figure}
We see that a highly oscillatory mode appears in the error along the
interfaces.  Snapshots of the error distribution using the lumped mass
matrix $B_i^{\textrm{lump}}$ are shown in
Figure~\ref{fig:FalseColorLumpErrorRepresentation} for the same
experiment setting.
\begin{figure}
  \centering
  \includegraphics[width=0.3\textwidth]{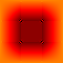}\quad
  \includegraphics[width=0.3\textwidth]{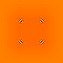}\quad
  \includegraphics[width=0.3\textwidth]{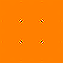}\\[1em]
  \includegraphics[width=0.3\textwidth]{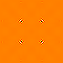}\quad
  \includegraphics[width=0.3\textwidth]{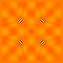}\quad
  \includegraphics[width=0.3\textwidth]{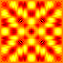}
  \caption{Scaled error distribution at iteration $35$, $50$, $65$,
    $80$, $95$, $110$ for OSM with lumped interface mass matrix using
    auxiliary variables at cross-points.}
  \label{fig:FalseColorLumpErrorRepresentation}
\end{figure}
We see that with the lumped mass matrix, the high frequency error mode
along the interface is much less pronounced, and convergence is
faster.

In order to understand this phenomenon, we reinterpret the effect
of mass lumping at the continuous level: the difference
\begin{equation*}
B^\mathrm{\textrm{lump}}_{i;j,j'} - B_{i;j,j'}
=\begin{cases}
\frac{p}{6}\sum_{j''}\lvert\bm{x}_{i;j}-\bm{x}_{i;j''}\rvert
&\text{if $j'=j$ and $\bm{x}_{i;j}$ lies on $\partial\Omega_i$},\\
-\frac{p}{6}\lvert\bm{x}_{i;j}-\bm{x}_{i;j'}\rvert&\text{if
  $\lbrack\bm{x}_{i;j}\bm{x}_{i;j'}\rbrack$ is an edge of  $\partial\Omega_i$},\\
0&\text{otherwise},
\end{cases}
\end{equation*}
looks like the discretization of a negative, one-dimensional
Laplacian. This holds technically only if the step size $h$ is constant
and we are not at a cross-point. In that case, the lumped matrix
actually discretizes the higher order transmission condition
\begin{equation*}
\frac{\partial u}{\partial\bm{n}_i}+\frac{ph^2}{6}\frac{\partial^2
  u}{\partial^2 \bm{\tau}}+pu.
\end{equation*}
If we could modify the value of $ph^2$, we would obtain a truly
optimizable higher order, or Ventcell, transmission condition.  This
motivates the idea of overlumping: introducing a relaxation parameter
$\omega$, we define
\begin{equation}
  B^{\omega}_{i;j,j'}:=(1-\omega)B_{i;j,j'}+\omega
  B^{\mathrm{lump}}_{i;j,j'},
\end{equation}
and thus obtain a discretization of the transmission condition
\begin{equation}\label{eq:OvLumpSecondOrder}
  \frac{\partial u}{\partial\bm{n}_i}
    +\omega\frac{ph^2}{6}\frac{\partial^2 u}{\partial^2 \bm{\tau}}+pu.
\end{equation}

We perform now a numerical experiment with this overlumped mass
matrix. For a rectangular domain $\Omega=(0,4)\times(0,2)$ with two
square subdomains $\Omega_1=(0,2)\times(0,2)$ and
$\Omega_2=(2,4)\times(0,2)$, we run the OSM on Laplace's equation
discretized with $Q_1$ finite elements and homogeneous boundary
conditions, thus simulating directly the error equations.  We start
with a random initial guess on the interface $\{2\}\times(0,2)$. We
apply $50$ Optimized Schwarz iterations.  We do this for $10\times10$,
$20\times20$, $50\times50$ and $100\times100$ cells per subdomains,
with the Robin parameter $p$ going from $1$ to $20$ with increment of
$0.5$ and the lump parameter $\omega$ going from $0$ to $100$ with
increment of $0.25$. 
We give the optimal $p$ and
$\omega$ in Table~\ref{table:OptimalFactorsOverlump}.
\begin{table}
\begin{tabular}{p{1.8cm}|p{3cm}|p{3cm}|p{3.5cm}}
Cells in $\Omega_i$&Consistent&Lumped&Best\\
\hline
$10\times10$  & 
    $\omega=0.0$,  $p=6.0$,  $\kappa= 0.5791628$  &
     $\omega=1.0$,    $p=3.5$,    $\kappa=0.3887587$  &
     $\omega=10.25$,    $p=1.5$,    $\kappa=0.1245496$ \\
\hline
$20\times20$  &
     $\omega=0.0$,    $p=8.5$,    $\kappa= 0.6853493$  &
      $\omega=1.0$,   $p=5.0$,    $\kappa=0.5222360$  &
      $\omega=17.75$,   $p=2.0$,     $\kappa=0.1852617$   \\
\hline
$50\times50$    &  
    $\omega=0.0$, $p=14.0$, $\kappa=0.7847913$ &
    $\omega=1.0$, $p=8.0$, $\kappa=0.6643391$ &
    $\omega=45.0$, $p=2.5$, $\kappa=0.2863597$ \\
\hline
$100\times100$  & 
$\omega=0.0$,    $p=22.5$,   $\kappa=0.8141025$ &
$\omega=1.0$,    $p=12.0$,    $\kappa=0.7332624$ &
$\omega=89.25$,    $p=3.0$,  $\kappa=0.3571062$
\end{tabular}
\caption{Optimal Robin parameter $p$ and overlump factor $\omega$ with
  corresponding numerical convergence factor $\kappa=\exp(\log(\lVert
  u_{50}\rVert_\infty/\lVert u_{0}\rVert_\infty)/50)$ and $2$ subdomains.}
\label{table:OptimalFactorsOverlump}
\end{table}
Using the asymptotic results from~\cite{Gander:2006:OSM}, the optimal
asymptotic choice of $p$ for the consistent mass interface matrix
should behave like $p=O(1/h^{1/2})$, and in the emulated Ventcell case
from overlumping, we should have $p=O(1/h^{1/4})$ and
$\omega=O(1/h)$, which is well what we observe.

We perform now a new numerical experiment with this overlumped mass
matrix but in the presence of a single cross-point. For this experiment, we
use the auxiliary variable method, see
Table~\ref{table:OptimalFactor2x2AV}, 
and complete communication\footnote{Using
$\mathbf{A}_{\mathcal{D}}$ and $\mathbf{A}_{\mathcal{N}}$ 
of~\S\ref{subsect:SplittingNeumann}
and~\S\ref{subsect:SplittingNeumann}}, see
Table~\ref{table:OptimalFactor2x2CC}. For a square domain $\Omega=(0,4)\times(0,4)$ with four
square subdomains $\Omega_1=(0,2)\times(0,2)$ and
$\Omega_2=(2,4)\times(0,2)$,  $\Omega_3=(0,2)\times(2,4)$ and
$\Omega_4=(2,4)\times(2,4)$, we run the OSM on Laplace's equation
discretized with $Q_1$ finite elements and homogeneous boundary
conditions, thus simulating directly the error equations.  We start
with a random initial guess on the interface $\{2\}\times(0,4)\cup(0,4)\times\{2\}$. We
apply $50$ optimized Schwarz iterations.  We do this for $10\times10$,
$20\times20$, $50\times50$ and $100\times100$ cells per subdomains.
We started with the Robin parameter $p$ going from $1$ to $20$ with increment of
$0.5$ and the lump parameter $\omega$ going from $0$ to $100$ with
increment of $0.25$. For the $100\times100$ cells per subdomain with
consistent Robin conditions case, we extended the search for the Robin
parameter up to $24.5$. For the best (overlumping) case, $2\times2$
subdomains and 
$10\times10$ cells per subdomain, we extended the search for the
optimal $p$ to the interval $\lbrack0.1,1\rbrack$ with  increment of $0.1$.
\begin{table}
\begin{tabular}{p{1.8cm}|p{3cm}|p{3cm}|p{3.5cm}}
Cells in $\Omega_i$&Consistent&Lumped&Best\\
\hline
$10\times10$  & 
    $\omega=0.0$,  $p=3.5$,  $\kappa=0.7468911$  &
     $\omega=1.0$,    $p=2.0$,    $\kappa=0.6833862$  &
     $\omega=17.25$,    $p=0.8$,    $\kappa=0.4862979$ \\
\hline
$20\times20$  &
     $\omega=0.0$,    $p=5.0$,    $\kappa=0.8073780$  &
      $\omega=1.0$,   $p=3.0$,    $\kappa=0.7053783$  &
    $\omega=14.75$, $p=1.5$, $\kappa=0.5045374$ \\
\hline
$50\times50$    &  
    $\omega=0.0$, $p=8.0$, $\kappa=0.8775996$ &
    $\omega=1.0$, $p=4.5$, $\kappa=0.8032485$ &
     $\omega=82.0$,   $p=1.5$,     $\kappa=0.5001431$   \\
\hline
$100\times100$  & 
$\omega=0.0$,    $p=11.0$,   $\kappa=0.9102802$ &
$\omega=1.0$,    $p=6.5$,    $\kappa=0.8547884$ &
$\omega=122.5$,    $p=2.0$,  $\kappa=0.6013464$     
\end{tabular}
\caption{Optimal Robin parameter $p$ and overlump factor $\omega$ with
  corresponding numerical convergence factor $\kappa=\exp(\log(\lVert
  u_{60}\rVert_\infty/\lVert u_{30}\rVert_\infty)/30)$ for $2\times2$
  subdomains using auxiliary variable method.}
\label{table:OptimalFactor2x2AV}
\end{table}

\begin{table}
\begin{tabular}{p{1.8cm}|p{3cm}|p{3cm}|p{3.5cm}}
Cells in $\Omega_i$&Consistent&Lumped&Best\\
\hline
$10\times10$  & 
    $\omega=0.0$,  $p=3.5$,  $\kappa=0.7553129$  &
     $\omega=1.0$,    $p=2.0$,    $\kappa=0.6967638$  &
     $\omega=17.75$,    $p=1.0$,    $\kappa=0.3989268$ \\
\hline
$20\times20$  &
     $\omega=0.0$,    $p=5.0$,    $\kappa=0.8134911$  &
      $\omega=1.0$,   $p=3.0$,    $\kappa=0.7082014$  &
      $\omega=15.0$,   $p=1.5$,     $\kappa= 0.4997952$   \\
\hline
$50\times50$    &  
    $\omega=0.0$, $p=8.0$, $\kappa= 0.8778605$ &
    $\omega=1.0$, $p=4.5$, $\kappa= 0.8034476$ &
    $\omega=86.0$, $p=1.5$, $\kappa=0.5141311$ \\
\hline
$100\times100$  &
$\omega=0.0$,    $p=11.0$,   $\kappa=0.9106798$ &
$\omega=1.0$,    $p=6.5$,    $\kappa=0.8528811$ &
$\omega=122.0$,    $p=2.0$, $\kappa=0.6006753$.
\end{tabular}
\caption{Optimal Robin parameter $p$ and overlump factor $\omega$ with
  corresponding numerical convergence factor $\kappa=\exp(\log(\lVert
  u_{60}\rVert_\infty/\lVert u_{30}\rVert_\infty)/30)$ for $2\times2$
  subdomains using complete communication method. Same $p$ at
  cross-point as on edge.}
\label{table:OptimalFactor2x2CC}
\end{table}

\section{A simple lemma on connected graphs}
\begin{lemma}\label{lemma:graphes}
Let $\mathcal{G}$ be a connected graph. Let $V(G)$ be its set of
vertices and $E(G)$ be its set of edges. Let $\phi$ be a function
from $V(G)$ to $\mathbb{R}$ such that $\sum_{v\in V(G)}\phi(v)=0$.
Let 
\begin{equation*}
E_f(G)=\{(v_1,v_2)\in V(G)\times V(G) \text{ s.t. }
\{v_1,v_2\}\in E(G)\}.
\end{equation*}
Then, there exists 
\begin{equation*}
\psi: E_f(G)\to\mathbb{R}
\end{equation*}
such that
\begin{align*}
\psi(v_1,v_2&)=-\psi(v_2,v_1)\text{ for all $(v_1,v_2)$ in $E_f(G)$},\\
\phi(v_1)&=\sum_{v_2\text{ s.t. $(v_1,v_2)$ in $E_f(G)$}} \psi(v_1,v_2).
\end{align*}
\end{lemma}
\begin{proof}
By recurrence over the number of vertices. The lemma is trivially true
when the number of vertices is $1$. Suppose the lemma is  true when
the number of vertices is $n$ with $n\geq 1$. Let $G$ be a connected graph with
$n+1$ vertices. It is well known that there exists a vertex $v$ such
that $G-\{v\}$ remains connected. Since $G$ is connected, there are
edges of $G$ originating from $v$. Choose $w_0$ adjacent to $v$. Set
$\psi(v,w_0):=\phi(v)$, $\psi(w_0,v):=-\phi(v)$ and
$\psi(v,w):=\psi(w,v):=0$ for all other vertices $w$ adjacent to $v$. Set
\begin{equation*}
\begin{split}
\hat{\phi}:V(G)\setminus \{v\}&\to\mathbb{R}\\
w&\mapsto\begin{cases}
\phi(w)&\text{if $w$ not adjacent to $v$},\\
\phi(w)-\psi(w,v)&\text{if $w$ adjacent to $v$}.
\end{cases}
\end{split}
\end{equation*}
We have $\sum_w\hat{\phi}(w)=\sum_w\phi(w)=0$.
We apply the lemma on $\hat{\phi}$ and $G-\{v\}$ which is connected and get the
remaining values of $\psi$.
\end{proof}

\section*{Acknowledgements}
This study has been carried out with financial support from the French
State, managed by the French National Research Agency (ANR) in the 
frame of the "Investments for the future" Programme IdEx Bordeaux -
CPU (ANR-10-IDEX-03-02).

\bibliographystyle{plain}
\bibliography{GanderSantugini_CrosspointsDDMFEM.bib}
\end{document}